\documentclass{sig-alt-full}

\usepackage[utf8]{inputenc}
\usepackage[T1]{fontenc}
\usepackage{lmodern}

\usepackage{comment}
\usepackage{amssymb,amsmath}
\usepackage{theorem}
\usepackage{graphics}
\usepackage{amsfonts}
\usepackage{mathrsfs}
\usepackage{amscd}
\usepackage{color}
\usepackage{url}

\newdef{defi}{Definition}
\newtheorem{thm}{Theorem}
\newtheorem{prop}{Proposition}
\newtheorem{cor}{Corollary}

\begin{document}
\conferenceinfo{ISSAC 2018,}{16--19 July 2018, New York, USA} 
\CopyrightYear{2018} 
\crdata{}

\title{A symplectic Kovacic's algorithm in dimension 4}
\newfont{\authfntsmall}{phvr at 11pt}
\newfont{\eaddfntsmall}{phvr at 9pt}

\numberofauthors{3} 
\author{
\alignauthor
Thierry Combot\\
       \affaddr{Univ. de Bourgogne (France)}\\
       \email{thierry.combot@u-bourgogne.fr}
\alignauthor
Camilo Sanabria\\
       \affaddr{Universidad de los Andes (Colombia)}\\
       \email{c.sanabria135@ uniandes.edu.co}	
}
\date{today}

\maketitle

\begin{abstract} 
Let $L$ be a $4$th order differential operator with coefficients in $\mathbb{K}(z)$, with $\mathbb{K}$ a computable algebraically closed field. The operator $L$ is called symplectic when up to rational gauge transformation, the fundamental matrix of solutions $X$ satisfies $X^t J X=J$ where $J$ is the standard symplectic matrix. It is called projectively symplectic when it is projectively equivalent to a symplectic operator. We design an algorithm to test if $L$ is projectively symplectic. Furthermore, based on Kovacic's algorithm, we design an algorithm that computes Liouvillian solutions of projectively symplectic operators of order $4$. Moreover, using Klein's Theorem, algebraic solutions are given as pullbacks of standard hypergeometric equations.
\end{abstract}

\bigskip
\noindent {\bf Categories and Subject Descriptors:} 34M15 \\


\noindent {\bf Keywords:} Differential Galois theory, Symplectic differential systems, Liouvillian functions, Kovacic's algorithm

\section{Introduction}

A solution to a linear ordinary differential equation is called Liouvillian over the base field if it is in a field extension obtained by successively adjoining antiderivative, hyperexponential or algebraic functions. The solutions to a linear ordinary differential equation are all Liouvillian if and only if the connected component of its differential Galois group is solvable. Existing algorithms computing Liouvillian solutions for linear ordinary differential equations with coefficients in $\mathbb{K}(z)$ are based on Kovacic algorithm \cite{KOVACIC1986,ULMER1996}. For equations of order $n$, these algorithms require a classification of  all Lie groups in $GL_n(\mathbb{K})$ whose identity components are solvable, a procedure to identify whether the differential Galois group of the equation corresponds to one of these Lie groups and in such case to which one, and finally an algorithm to effectively compute the solutions. The procedure of identifying the differential Galois group relies on computing the invariants and semi-invariants of the equation (i.e. the rational and hyperexponential solutions to the symmetric powers). The main problem is that, even if the complete list of possible groups is known up to dimension $4$, effective calculations of the invariants and semi-invariants is not available because identifying the differential Galois group requires computations of hyperexponential or rational solutions of very high order. Because of this, a fully effective algorithm for computing Liouvillian solutions has for now been restricted to order $2$, which is precisely the Kovacic algorithm. This article intends to extend the Kovacic algorithm up to order $4$, under the assumption that the equation is symplectic.

In dynamical systems, many physical problems involving conservative forces admit a Hamiltonian formulation, and therefore the conservation of a non-degenerated $2$-form, called the symplectic structure. When the system is linearised for analysis of perturbations, the symplectic structure defines a symplectic structure on the linearised system. In the study of the linearised system one often requires the computation of its solutions, and up to now, only in very specific cases, explicitly when the equations split in order $2$ equations, this solutions can be effectively computed using Kovacic's algorithm. Now, every symplectic system is even dimensional. In dimension $2$, symplectic systems are unimodular systems because $SP_2(\mathbb{K})=SL_2(\mathbb{K})$. In general, we have $SP_{2n}(\mathbb{K})\subset SL_{2n}(\mathbb{K})$ and this inclusion is strict for $n>1$. Moreover, for $n>1$, $SP_{2n}(\mathbb{K})$ properly contains $SL_{2}(\mathbb{K})^n$, and thus symplectic systems cannot be in general uncoupled in order $2$ systems.

\begin{defi}
A matrix $M\in GL_{2n}(\mathbb{K})$ is symplectic if it satisfies $M^tJM=J$ where
$$J=\left(\begin{array}{cc} 0&I_n\\ -I_n&0 \end{array}\right),$$
is the standard symplectic form. It is projectively symplectic if it satisfies $M^tJM=\lambda J$ for some $\lambda\in\mathbb{K}^*$. The group of symplectic matrices is denoted $SP_{2n}(\mathbb{K})$ and the group of projectively symplectic matrices $PSP_{2n}(\mathbb{K})$.
\end{defi} 

\begin{defi}
The Lie algebra of $SP_{2n}(\mathbb{K})$ and $PSP_{2n}(\mathbb{K})$ are respectively
$$\mathfrak{sp}_{2n}(\mathbb{K}) =\{M\in M_{2n}(\mathbb{K}), \; M^t J+JM=0\},$$
$$\mathfrak{psp}_{2n}(\mathbb{K})=\{M\in M_{2n}(\mathbb{K}), \; \exists \lambda\in\mathbb{K}, \; M^t J+JM=\lambda J\}.$$
\end{defi}

Let us recall that the Galois group of a differential operator $L$ of order $n$ with coefficients in $\mathbb{K}(z)$ is the group of differential automorphisms of the differential field generated by the solutions of $L$ fixing the base field $\mathbb{K}(z)$. This group is always isomorphic to a Lie subgroup of $GL_n(\mathbb{K})$.

\begin{defi}
An operator $L$ of order $2n$ is symplectic (respectively projectively symplectic) when its Galois group is isomorphic to a subgroup of $SP_{2n}(\mathbb{K})$ (resp. of $PSP_{2n}(\mathbb{K})$).
\end{defi}

Writing $L(x)=0$ as a differential system $X'=AX$ in dimension $n$, Kolchin-Kovacic Theorem ensures us that if the Galois group of $L$ is contained in a connected Lie group $G$, then there exists a $\mathbb{K}(z)$ gauge transformation such that the differential system writes $X'=AX$ with $A$ in the Lie algebra of $G$ \cite{APARICIO2013}. We can thus also use this property to equivalently define symplectic/projective symplectic operators

\begin{defi}
An operator $L$ of order $2n$ is symplectic, respectively projectively symplectic, if there exists an invertible matrix $P$ with coefficients in $\mathbb{K}(z)$ such that we have respectively
$$ P^{-1}AP+P'P \in \mathfrak{sp}_{2n}(\mathbb{K}(z)),\quad P^{-1}AP+P'P \in \mathfrak{psp}_{2n}(\mathbb{K}(z))$$
where $A$ is the companion matrix associated to $L$.
\end{defi}

Linear ordinary differential equations arising from problems in physics with a Hamiltonian formulation are symplectic, however the base field is not always the field of coefficients $\mathbb{K}(z)$. In many cases this is due to algebraic changes of the independent variable, which can introduce algebraic extensions of the base field. In particular, it is possible that the linear system is no longer symplectic, but its Galois group is included in a finite extension of $SP_{2n}(\mathbb{K})$ in $PSP_{2n}(\mathbb{K})$. However these systems are still projectively symplectic (see Proposition \ref{prop6}).
Throughout the paper we assume familiarity with the Picard-Vessiot theory and the tensorial constructions of differential modules \cite{VANDERPUT2003}, as well as with the algorithms to factor linear ordinary differential operators \cite{VANHOEIJ1997}, to find their rational \cite{BARKATOU1999} and hyperexponential solutions \cite{PFLUGEL1997} over $\mathbb{K}(z)$ or over a quadratic extensions of $\mathbb{K}(z)$ \cite{BURGER2004}.

\section{Checking symplecticity}

\begin{prop}\label{propsymp}
Let $L$ be a $2n$-th order operator with coefficients in $\mathbb{K}(z)$ and $A$ its companion matrix. The operator $L$ is symplectic (respectively projectively symplectic) if and only if there exists an invertible antisymmetric matrix $W$ with coefficients in $\mathbb{K}(z)$ such that
$$A^tW+WA+W'=0$$
and respectively for projective symplectic a $\lambda\in\mathbb{K}(z)$ such that
\begin{equation}\label{eqsymp}
A^tW+WA+W'+\lambda W=0.
\end{equation}
The matrix $W$ will be called a symplectic structure associated to $L$, and $\lambda$ its multiplier.
\end{prop}

\begin{proof}
Let $X=(x_{i,j})$ be a fundamental matrix of solutions of $L$. Let $V$ be the vector space over $\mathbb{K}$ generated by the columns of $X$. Let $G\subseteq GL_{2n}(\mathbb{K})$ be the representation of the Galois group of $L$ associated to $X$.

Assume first that $L$ is projectively symplectic. Since $L$ is projectively symplectic, we may choose $X$ so that $G\subseteq PSP_{2n}(\mathbb{K})$. Then $\tilde{W}=(X^{-1})^tJX^{-1}$ is an invertible antisymmetric matrix with semi-invariant coefficients. Indeed, for any $g\in G$, $(X^{-1})^t(g^{-1})^tJg^{-1}X^{-1}=(X^{-1})^t\lambda_gJX^{-1}=\lambda_g\tilde{W}$ for some $\lambda_g\in\mathbb{K}$. On the other hand $X^t\tilde{W}X=J$, therefore
\[
0=A^t\tilde{W}+\tilde{W}'+\tilde{W}A.
\]
The invertible antisymmetric matrix with semi-invariant coefficients $\tilde{W}$ defines a non-degenerate alternating bilinear forms on $V$, so the coefficients of $\tilde{W}$ form a hyperexponential solution to the exterior square of the dual of the system $X'=AX$ (for, if $\tilde{W}=(\tilde{w}_{i,j})$ and $J=(\jmath_{i,j})$, $$\jmath_{i,j}=\sum_{k,l=1}^{2n} x_{k,i}\tilde{w}_{k,l}x_{l,j}=\sum_{1=k<l=2n} \tilde{w}_{k,l}(x_{k,i}x_{l,j}-x_{k,j}x_{l,i}).)$$
So there exist $\lambda\in\mathbb{K}(z)$ such that $W=e^{-\int \lambda(z)dz}\tilde{W}\in\mathbb{K}(z)$ and
\[
0=A^tW+W'+\lambda W+WA.
\]

Conversely, assume that there exist an invertible antisymmetric matrix $W$ with coefficients in $\mathbb{K}(z)$ 
and a $\lambda\in\mathbb{K}(z)$ such that
\[
A^tW+WA+W'+\lambda W=0.
\]
Then, if $\tilde{W}=e^{\int \lambda(z)dz}W$, $X^t\tilde{W}X$ is a matrix with constant coefficients. In particular, for every $g\in G$ we have $g^tX^t\tilde{W}Xg=X^t\tilde{W}X$, so $G$ is isomorphic to a subgroup of $PSP_{2n}(\mathbb{K})$.

The non-projective case correspond to the case $\lambda=0$.
\end{proof}

\begin{cor}
Let $L$ be a $2n$-th order operator with coefficients in $\mathbb{K}(z)$. If the operator $L$ is symplectic (respectively projectively symplectic) then the exterior square of $L$ has a rational (respectively hyperexponential) solution.
\end{cor}

\begin{proof}
In the proof of proposition \ref{propsymp}, we showed that if $A$ is the companion matrix of $L$, the equation on $\tilde{W}$, $A^t\tilde{W}+\tilde{W}A+\tilde{W}'=0$, is equivalent to the exterior square of the dual of $X'=AX$. If the system is symplectic (respectively projectively symplectic), the equation $A^tW+WA+W'=0$ admits a rational (respectively hyperexponential) invertible matrix solution, and thus so its dual, which is equivalent to the exterior square of $L$.
\end{proof}

A function is hyperexponential over $\mathbb{K}(z)$ if its logarithmic derivative is in $\mathbb{K}(z)$. The exponential type of $f$ is the equivalent class of $f$ for the equivalence relation
$$f\sim g \Leftrightarrow f/g \in\mathbb{K}(z).$$

\begin{defi}
A non-trivial solution of $A^tW+WA+W'=0$ which is
\begin{itemize}
\item hyperexponential is called a projective Poisson structure.
\item hyperexponential invertible is called a projective symplectic structure.
\item rational is called a Poisson structure.
\item rational invertible is called a symplectic structure.
\end{itemize}
\end{defi}

Note that if $W,\lambda$ is a solution in the projective case, then
$$e^{-\int \lambda(z) dz}W(z)$$
is a solution in the non-projective case whenever $\exp{\int \lambda(z) dz}$ $\not\in\mathbb{K}(z)$. In particular we can detect whether the Galois group is a finite extension of $SP_{2n}(\mathbb{K})$.

\begin{prop}
Let $L$ be a $2n$-th order operator with coefficients in $\mathbb{K}(z)$ and $A$ its companion matrix. The Galois group of $L$ is isomorphic to a subgroup of $\mathbb{Z}_p \ltimes SP_{2n}(\mathbb{K})$ if and only if it is projectively symplectic with $\lambda\in\mathbb{K}(z)$ as in Proposition \ref{propsymp} such that
$$e^{p\int \lambda(z) dz} \in\mathbb{K}(z).$$
\end{prop} 

\begin{proof}
Let $L$ be projectively symplectic. Let $X$ and $\tilde{W}=\exp{\int \lambda(z) dz} W=(X^{-1})^tJX^{-1}$ be defined as in the proof of Proposition \ref{propsymp}. Since $\tilde{W}$ is semi-invariant under the action of the Galois group $G$, for every $\sigma\in G$ there exist $\lambda_{\sigma}\in\mathbb{K}$ such that $\sigma(\exp{\int \lambda(z) dz})=\lambda_{\sigma}\exp{\int \lambda(z) dz}$. So $\lambda\in\mathbb{K}(z)$ such that $\exp{p\int \lambda(z) dz} \in\mathbb{K}(z)$ is equivalent to $\exp{(p\int \lambda(z) dz)}=\sigma(\exp{(p\int \lambda(z) dz)})=\lambda^p_{\sigma} \exp{(p\int \lambda(z) dz)}$, or equivalently $\lambda_{\sigma}\in \mathbb{U}_p$. Now, if $g$ is the image of $\sigma$ in the representation associated to $X$ of $G$ in $GL_{2n}(\mathbb{K})$, we have $g^tJg=\lambda_\sigma J$. Therefore $\exp{(p\int \lambda(z) dz)} \in\mathbb{K}(z)$ is equivalent to $g\in \mathbb{Z}_p \ltimes SP_{2n}(\mathbb{K})$. 
\end{proof}

A symplectic structure is an antisymmetric matrix of dimension $2n$, and thus we can compute its Pfaffian. 

\begin{defi}
Let $W=(w_{i,j})$ be an antisymmetric matrix of dimension $2n$. The Pfaffian of $W$ is defined by
$$\hbox{Pf}(W)=\frac{1}{2^nn!}\sum_{\sigma\in S_{2n}}\hbox{sgn}(\sigma)\prod_{i=1}^n w_{\sigma(2i-1),\sigma(2i)}.$$
\end{defi}

Recall that the square of the Pfaffian is the determinant. From the relation $X^t \tilde{W} X =J$ in the proof of Proposition \ref{propsymp} we get $\hbox{Pf}(X^t \tilde{W} X) =\hbox{Pf}(J)=1$ and thus $\hbox{Pf}(\tilde{W})=\hbox{det}(X)^{-1}$. Hence
$$\hbox{Pf}\left(e^{-\int \lambda(z) dz} W\right)=\hbox{det}(X)^{-1}$$
and
$$e^{n\int \lambda(z) dz} \hbox{Pf}\left(W^{-1}\right)=\hbox{det}(X).$$
As $\hbox{Pf}(W^{-1}) \in\mathbb{K}(z)$, we can deduce the exponential type of $\exp{n\int \lambda(z) dz}$, but not  the exponential type of $\exp{\int \lambda(z) dz}$, as possibly $n$-th roots could appear. In fact
$$\{M\in GL_{2n}(\mathbb{K}), \; \exists \lambda\in\mathbb{U}_n, \; M^t J M=\lambda J\} \subset SL_{2n}(\mathbb{K}).$$
So in particular, even if the Wronskian of the system is a constant and the system is projectively symplectic, the Galois group is either symplectic or a subgroup of $\mathbb{Z}_n \ltimes SP_{2n}(\mathbb{K})$.

Our algorithm to detect symplecticity has two parts. The first part searches for solutions to the exterior square of the dual system. The second part identifies the solutions that define a non-degenerate antisymmetric form and their exponential type.\\

\noindent\underline{\sf IsSymplectic}\\
\textsf{Input:} A differential operator $L$ of order $2n$ with coefficients in $\mathbb{K}(z)$.\\
\textsf{Output:} A projective symplectic structure if it exists.\\
\begin{enumerate}
\item Noting $W$ the matrix
$$\left(\begin{array}{ccccc} 0& J_1(z) & J_2(z) & \cdots & J_{2n-1}(z)\\-J_1(z)&0& J_{2n}(z)& \cdots &  J_{4n-3}(z)\\ -J_2(z) & -J_{2n-1}(z) & 0 & \cdots & J_{6n-6}(z)\\ \vdots & \vdots & \vdots & \ddots & \vdots \\ -J_{2n-1}(z) & -J_{4n-3}(z) & -J_{6n-6}(z) & \cdots & 0 \\ \end{array} \right)$$
and $A$ the companion matrix of $L$, we write down the system
\begin{equation}\label{eqsympext}
A^tW+WA+W'=0.
\end{equation}
\item Compute a basis $B=\{\tilde{W}_1,\ldots,\tilde{W}_m\}$ of the hyperexponential solutions of (\ref{eqsympext}), which is equivalent to the exterior square of the dual of $X'=AX$.
\item For each exponential type of solution in $B$, look for linear combinations of the $\tilde{W}_i$'s with the same exponential type, $\tilde{W}_{i_1},\ldots,\tilde{W}_{i_p}$, such that $\hbox{Pf}(a_1\tilde{W}_{i_1}+\ldots+a_p\tilde{W}_{i_p})\neq 0$. If there are none, return $[]$. Else set $\tilde{W}=a_1\tilde{W}_{i_1}+\ldots+a_p\tilde{W}_{i_p}$.
\item Let $\exp(\int \lambda(z) dz)$ be the exponential type of $\tilde{W}$. Set $W=\exp(-\int \lambda(z) dz)\tilde{W}$
\item Return $\exp(\int \lambda(z) dz)W(z)$.
\end{enumerate}

\begin{thm}
The algorithm \underline{\sf IsSymplectic} returns a solution if and only if $L$ is projectively symplectic.
\end{thm}

\begin{proof}
The solutions $W_i$ form a basis of the projective Poisson structures (step 1). A projective symplectic structure is a hyperexponential solution of \eqref{eqsympext}, and thus is a linear combination of the $W_i$. As it should be hyperexponential, it can only be a linear combination of $W_i$ of the same type and the invertible condition is equivalent to a non zero Pfaffian (step 2). The symplectic structure and its type is returned (steps 3,4).
\end{proof}

\section{A Kovacic algorithm}

From now on, we will restrict ourselves to order $4$ operators. Let us first classify the possible Lie groups we will encounter.

\subsection{Symplectic groups}

\begin{defi}
A Lie subgroup $G$ of $SL_4(\mathbb{K})$ is said to be intransitive if it admits a block triangular representation, i.e. if it stabilizes a non-trivial proper subspace of $\mathbb{K}^4$, otherwise it is called transitive. It is called imprimitive if there exists a decomposition $\mathbb{K}^4=\oplus_{i=1}^p V_i$, where each $V_i$ is a non-trivial proper subspace, such that $G$ acts permuting the $V_i$'s, i.e.
$$\forall g\in G, \exists \sigma \in S_p \textrm{ s.t. } g(V_i) \subset V_{\sigma(i)}, \forall  i=1\dots p.  $$
Otherwise it is called primitive. When $G$ is imprimitive and $p=4$, $G$ will be called monomial imprimitive.
\end{defi}

\begin{prop}\label{propgroups}
A Lie subgroup of $SP_4(\mathbb{K})$ is up to conjugacy generated by elements of the form
\begin{itemize}
\item[i)] Upper block triangular matrices with diagonal blocks of size at most $2\times 2$.
\item[ii)] $2\times 2$ block diagonal matrices and anti-diagonal matrices.
\item[iii)] The group $SP_4(\mathbb{K})$.
\end{itemize}
\end{prop}

\begin{proof}
We recall that a group $G\subseteq SL_4(\mathbb{K})$ is conjugate to a subgroup of the symplectic group if we can find an antisymmetric invertible matrix $J_0\in GL_4(\mathbb{K})$ such that
$$g^tJ_0g=J_0,\quad \forall g\in G.$$
We will prove the proposition following the classification of Lie subgroups of $SL_4(\mathbb{K})$ found in \cite{HARTMANN2002}, by first examining the transitive primitive groups, then the transitive imprimitive non-monomial followed by the transitive imprimitive monomial and finally the intransitive.

Let us begin by looking for transitive primitive groups. We just have to test the above condition on a generating set of the group. Thus the condition becomes a set of linear equations and the entries of $J_0$ and an inequality condition $\hbox{Pf}(J_0)\neq 0$.

We begin with the $30$ finite primitive groups which are listed in \cite{HANANY2001}. We find that none of them is in $SP_4(\mathbb{K})$. 

Now for infinite primitive groups, there are four possibilities with $G^0$ irreducible, namely $G^0\simeq SL_4(\mathbb{K}),SP_4(\mathbb{K}),SO_4(\mathbb{K})$ and $SL_2(\mathbb{K})$ in its third symmetric power representation, and the only one satisfying the condition is $SP_4(\mathbb{K})$. We obtain case iii) of the proposition.

The possibilities for infinite primitive groups with $G^0$ reducible are the groups $G$ with
$$G^0=\left\lbrace \left( \begin{array}{cc} A & 0\\ 0 & A \end{array}\right), A\in SL_2(\mathbb{K}) \right\rbrace$$
and $G=HG^0$ where $H=I_2 \otimes T$ and $T$ is a finite subgroup of $SL_2(\mathbb{K})$. The symplectic structures compatible with such $G^0$ are of the form
$$\left(\begin{array}{cccc} 0& u_1 & 0& -u_4\\ -u_1& 0 & u_4 & 0 \\ 0& -u_4 & 0 & u_6\\ u_4 & 0 & -u_6 & 0 \end{array} \right)$$
Now the possible groups for $H$ are diagonal extensions of order at most $2$ of the primitives groups $A_4,S_4,A_5$. Let us search for $2\times 2$ matrices $T$ of determinant $1$ leading to a matrix $H$ compatible with one of the above symplectic structures. If $\hbox{det}(T)=1$ and $T$ is not triangular, we find
$$T=\left( \begin{array}{cc}
\frac{qt^2-t^2w+q+w}{2qt} & -\frac{(q^2-w^2)(t^2-1)}{4qt}\\ -\frac{t^2-1}{qt} & \frac{qt^2+t^2w+q-w}{2qt} \end{array}\right) $$
where $q,w$ are fixed constants and $t\in\mathbb{K}$ is free. These matrices generate either a finite cyclic group if $t\in\mathbb{U}_n$, a group isomorphic to $\mathbb{K}^*$, or when adding a matrix $T$ with $\hbox{det}(T)=-1$ a degree two extension of these two. So the only finite groups generated by matrices of type $T$ compatible with the symplectic structure are either cyclic or finite dihedral and thus the groups $A_4,S_4,A_5$ and their diagonal extensions are not possible. Thus this case cannot happen.

Now we consider the transitive imprimitive  non-monomial groups. They exchange two $2$-dimensional vector spaces $V_1$, $V_2$. On an adapted basis, the group is then generated by $2\times 2$ block diagonal matrices and anti-diagonal matrices
$$\left(\begin{array}{cccc} \star & \star &0 &0\\ \star & \star &0 &0 \\ 0&0 & \star&\star\\ 0&0&\star &\star\\ \end{array}\right),\left(\begin{array}{cccc} 0 & 0 &\star &\star\\ 0 & 0 &\star &\star \\ \star & \star & 0&0\\ \star & \star &0&0\\ \end{array}\right).$$
We obtain case ii) of the proposition.

For transitive imprimitive monomial groups, we need to look for matrices of the form $\hbox{diag}(a,b,c,d)P$ where $P$ is a permutation matrix of $S_4$. The permutation group acting on four $1$-dimensional vector spaces should be transitive, and thus can be one of the following groups:
\begin{itemize}
\item[a)] $\mathbb{Z}_4\simeq\langle(1,2,3,4)\rangle$,
\item[b)] the Klein group in it standard representation (i.e. generated by $(1,2)(3,4),(1,3)(2,4)$),
\item[c)] the dihedral group $D_8\simeq\langle (1,2,3,4),(1,3) \rangle$,
\item[d)] $A_4=\langle (1,2,3),(1,2)(3,4) \rangle$, and,
\item[e)] $S_4$.
\end{itemize}
For the Klein group, the group then admits a representation by $2\times 2$ block diagonal and anti-diagonal matrices. For the groups $\mathbb{Z}_4$ and $D_8$ we can consider the decomposition
$$\{(1,0,0,0),(0,0,1,0)\} \oplus \{(0,1,0,0),(0,0,0,1)\}$$
on which the group either permutes or stabilizes the two $2$-dimensional spaces. For these three groups we obtain case ii) of the proposition.

The representation of $A_4$
$$\left\langle \left(\begin{array}{cccc} 0&1&0&0\\ 0&0&1&0\\1&0&0&0\\ 0&0&0&1\end{array}\right), \left(\begin{array}{cccc} 0&1&0&0\\ -1&0&0&0\\0&0&0&1\\ 0&0&-1&0\end{array}\right) \right\rangle $$
is compatible with the symplectic structure
$$J=\left(\begin{array}{cccc} 0&1&-1&-1\\ -1&0&1&-1\\1&-1&0&-1\\ 1&1&1&0\end{array}\right).$$
The subspace $V$ is spanned by $v_1$ and $v_2$ where
$$v_1=(-\sqrt{3}-i,i\sqrt{3}-1,-2-\sqrt{3}+i,1+2i+i\sqrt{3})^t$$
and
$$v_2=(i+i\sqrt{3},-1-\sqrt{3},1-i,1+i)^t$$
is stable under this representation. The complex conjugate of $V$, $\bar{V}$, is also stable by $G$ and $V\oplus \bar{V}=\mathbb{K}^4$. Therefore this group is block diagonalizable with $2\times 2$ blocks. We obtain case i) of the proposition.

The representation of $S_4$
$$G=\left\langle \left(\begin{array}{cccc} 0&-i&0&0\\0&0&-i&0\\0&0&0&-i\\i&0&0&0\\\end{array}\right), 
\left(\begin{array}{cccc} 0&-i&0&0\\-i&0&0&0\\0&0&-i&0\\0&0&0&i\\\end{array}\right) \right\rangle .$$
is a degree $2$ extension of the previous case, and thus after conjugacy, the group is generated by $2\times 2$ block diagonal matrices and anti-diagonal matrices. We obtain case ii) of the proposition.

Now we treat the case where $G$ is intransitive. We assume the group is symplectic. Up to conjugacy we may assume that $G$ is such that $g^tJg=J$ for every $g\in G$, where $J$ is the standard symplectic form. Since $G$ is intransitive, the group has a stable subspace $V$ of dimension either $1$, $2$ or $3$ where the group acts transitively. Assume first the case $\dim(V)=3$ and let $V^J$ be the symplectic orthogonal complement of $V$. We have $\dim(V^J)=1$, and let $w\in V^J$, $w\ne 0$. Let $g\in G$, then for every $v\in V$, if we denote $v'=g^{-1}v\in V$, we have $$(gw)^tJv=w^tg^tJgv'=w^tJv'=0.$$
Therefore $V^J$ is stable by $G$. But since $\dim(V^J)=1$, $V^J\subseteq (V^J)^J=V$, which contradicts the assumption $G$ acts transitively on $V$. Similarly, in the case $\dim(V)=1$, $V^J$ is stable under $G$ and $V\subseteq V^J$, therefore we can represent $G$ by block triangular matrices with blocks in the diagonal of size at most $2\times 2$. We have case i) of the proposition. Finally, if $\dim(V)=2$, we also have case i).
\end{proof}

We deduce that if a symplectic $L$ is irreducible but has Liouvillian solutions (case ii) in Proposition \ref{propgroups}), then the Galois group is of the form $\mathbb{Z}_2 \ltimes G_1$ with $G_1$ subgroup of $SL_2(\mathbb{K})$. In particular, in that case $L$ admits a factorization in two operators of order $2$ with coefficients in a quadratic extension of $\mathbb{K}(z)$.

\subsection{The irreducible solvable case}

If $L$ is symplectic irreducible with Liouvillian solutions, we show that the factorization of $L$ in two operators of order $2$ with a quadratic extension can be identified using Poisson structures.

\begin{prop}\label{prop1}
If the Galois group of a symplectic operator $L$ can be represented as subgroup of $SP_4(\mathbb{K})$ formed by $2\times 2$ block diagonal and anti-diagonal matrices, then $L$ admits two linearly independent rank two Poisson structures with coefficients in a quadratic extension of $\mathbb{K}(z)$. 
\end{prop}

\begin{proof}
The Galois group of $L$ is isomorphic to $\mathbb{Z}_2 \ltimes G_1$ where $G_1$ is a subgroup of $SL_{2}(\mathbb{K})$. Let $X=(x_{i,j})$ be a fundamental matrix of solutions of $L$. Let $V$ be the vector space over $\mathbb{K}$ generated by the columns of $X$. Since $L$ is symplectic there exits an invertible antisymmetric matrix $W\in GL_4(\mathbb{K}(z))$ such that $X^tWX=J$. Let $K=\mathbb{K}(z)(x_{i,j})$ be the Picard-Vessiot extension of $L$ and $K_1=K^{G_1}$. We have that $K_1$ is a quadratic extension of $\mathbb{K}(z)$ and the Galois group of $L$ over $K_1$ is isomorphic to $G_1$. Therefore $V=V_1\oplus V_2$ where $V_1$ and $V_2$ are symplectic subspaces invariant under $G_1$ and there exists $4\times 4$ projection matrix $P_1$ and $P_2$ with coefficients in $K_1$ such that $P_1+P_2=I_4$ and the columns of $P_1X$ and $P_2X$ span $V_1$ and $V_2$ respectively. Define $W_1=P_1^tWP_1$ and $W_2=P_2^tWP_2$.  The anti-symmetric matrices $W_1$ and $W_2$ define two linearly independent rank two Poisson structures with coefficients in a quadratic extension of $\mathbb{K}(z)$.
\end{proof}

Note that the two Poisson structures in the proposition correspond to two conjugate solutions in a quadratic extension of the exterior square of $L$.

\begin{prop}\label{prop2}
The kernel of a Poisson structure $W$ of a symplectic operator $L$ is an invariant vector space.
\end{prop}

\begin{proof}
Let $X=(x_{i,j})$ be a fundamental matrix of solutions of $L$ and let $V$ be the vector space over $\mathbb{K}$ generated by the columns of $X$. The Galois group of $L$ acts on $V$ and it fixes $W$, for the Poisson structure has coefficients in the ground field. In particular, the Galois group stabilizes the kernel of the Poisson structure.
\end{proof}

\begin{cor}\label{cor}
\begin{itemize}
\item[i)] If $L$ is symplectic and irreducible, then all projective Poisson structures are projective symplectic structures and their Pfaffians have the same exponential type.
\item[ii)] An irreducible symplectic operator $L$ whose Galois group is isomorphic to $\mathbb{Z}_2 \ltimes G_1,\; G_1\subset SL_2(\mathbb{K})$ admits two linearly independent projective symplectic structures in a quadratic extension of $\mathbb{K}(z)$.
\item[iii)] Let $L$ be an irreducible operator admitting a symplectic structure $W_1$ and $W_2$ a linearly independent projective symplectic structure. Then there are two linear combinations of $W_1,W_2$ that are linearly independent rank two Poisson structures in a quadratic extension of $\mathbb{K}(z)$.
\end{itemize}
\end{cor}

\begin{proof}
\begin{itemize}
\item[i)] If $L$ admits a projective Poisson structure that is not symplectic, then, after multiplication by a hyperexponential function, we may assume that the Poisson structure has rational coefficients. Now from Proposition \ref{prop2} it follows that its kernel is a non-trivial invariant vector space of dimension $2$, and thus $L$ is reducible. Thus if $L$ is irreducible, then all projective Poisson structure are symplectic structures. Now, since a solution to equation \eqref{eqsympext} has a Pfaffian equal, up to a constant factor, to the Wronskian of $L$ and the Pfaffian is not zero, they all have the same exponential type (the exponential type of the Wronskian).
\item[ii)] If $L$ is an irreducible symplectic operator with Galois group isomorphic to $\mathbb{Z}_2 \ltimes G_1$ where $G_1$ is a subgroup of $SL_2(\mathbb{K})$, then from Proposition \ref{prop1}, $L$ admits two linearly independent Poisson structures in a quadratic extension. In particular, it admits two projective Poisson structures and therefore, from i), these are projective symplectic structures.
\item[iii)] Let $W_1,\ldots,W_p$ be a basis of the projective symplectic structures, with $W_1$ rational. From i) we know all their Pfaffian have the same exponential type, and thus the Pfaffians of $W_2,\dots,W_p$ are rational. Thus the $W_i$ belong to a field extension $\mathbb{F}=\mathbb{K}(z,\sqrt{w_2},\dots,\sqrt{w_{p}})$, $w_i\in\mathbb{K}(z)$. So the Galois group $G=\hbox{Gal}(L/\mathbb{F})$ admits $p$ invariant symplectic structures $J_1,\dots,J_p$. It also admits Poisson structures which are the linear combinations of the $J_i$ with zero Pfaffian. The space of non invertible Poisson structures stabilized by $G$ is thus of dimension $p-1$. Thus $L$ should admit a continuum of dimension $p-1$ of non invertible Poisson structures in $\mathbb{F}$. A solution $W$ of \eqref{eqsympext} in $\mathbb{F}$ can be written as a linear combination of elements in $\sqrt{w_i}M_4(\mathbb{K}(z))$ and $M_4(\mathbb{K}(z))$. Acting the Galois group on the square roots concludes that each of these matrices are themselves solutions of \eqref{eqsympext}, and thus that $W$ is a linear combination of the $W_i$. Thus the $W_i$ form a basis of the Poisson structures with coefficients in $\mathbb{F}$. Therefore the equation
\begin{equation}\label{eqPoisson}
\hbox{Pf}\left(\sum\limits_{i=1}^p \lambda_i W_i \right)=0 \in \mathbb{K}(z)
\end{equation}
admits a continuum of solution of dimension $p-1$, and so is equivalent to a single equation in $(\lambda_1:\ldots:\lambda_p)\in\mathbb{P}^{p-1}(\mathbb{K})$. Bezout theorem implies that this projective variety intersects $\lambda_3=\dots=\lambda_p=0$, and thus that a linear combination of $\mu_1W_1+\mu_2 W_2$ gives a rank two Poisson structure. Now if $W_2$ is rational, then the resulting Poisson structure has coefficients in $\mathbb{K}(z)$, which would mean that $L$ is reducible, contrary to the hypothesis. Thus $W_2$ is not rational, and its exponential type has a square root. So by Galois action, $\mu_1W_1-\mu_2 W_2$ is also a rank two Poisson structure. These are distinct as $\mu_1,\mu_2\neq 0$, because $W_1,W_2$ are invertible.
\end{itemize}
\end{proof}

Example:
$$G=\left\langle \! \left(\begin{array}{cccc} \lambda & 0 &0 &0\\ 0 & 1/\lambda &0 &0 \\ 0&0 & \epsilon\lambda&0\\ 0&0&0&\epsilon/\lambda\\ \end{array}\right),\left(\begin{array}{cccc} 0 & 1 &0 &0\\ -1 & 0 &0 &0 \\ 0&0 & 0&1\\ 0&0&-1&0\\ \end{array}\right), \left(\begin{array}{cccc} 0 & 0 &1 &0\\ 0 & 0 &0 &1 \\ 1&0 & 0&0\\ 0&1&0&0\\ \end{array}\right)\! \right\rangle,$$
$\lambda\in\mathbb{K}^*,\epsilon=\pm 1$. This group is transitive, admits two subgroups of index $2$, more precisely generated by removing the third matrix, and restricting $\epsilon=1$. Each of these subgroups admits two symplectic structures, and the intersection, an index $4$ subgroup, admits three symplectic structures.

\section{The algorithm}

Note that if $L$ is projectively symplectic, then up to a multiplication of a hyperexponential function, we can ensure that the operator is then symplectic. Indeed, using the notation in the proof of Proposition \ref{propsymp}, if $X^t \exp({\int \lambda(z) dz})W X=J$, then for $Y=\exp({\frac{1}{2}\int \lambda(z) dz})X$, we have $Y^tWY=J$.

We know that if $L$ has Liouvillian solutions, then we can always obtain a factorization with operators of order $2$, either over the base field or in quadratic extension of it. To obtain an explicit expression of the Liouvillian solutions, we apply the Kovacic algorithm to the order two factors: More precisely the Ulmer-van Hoeij-Weil version of Kovacic algorithm \cite{ULMER1996,VANHOEIJ2005}. Indeed, the important property of it is that it only uses the computation of rational solutions of symmetric powers, which will be here much easier than the computation of hyperexponential solutions for an operator with non-rational coefficients.\\

\noindent\underline{\sf SymplecticKovacic}\\
\textsf{Input:} A symplectic differential operator $L=\partial^4+a(z)\partial^3+b(z)\partial^2+c(z)\partial+d(z)\in\mathbb{K}(z)[\partial]$.\\
\textsf{Output:} A basis of the vector space of Liouvillian solutions of $L$.\\
\begin{enumerate}
\item Factorize $L$ in $\mathbb{K}(z)[\partial]$ \cite{VANHOEIJ1997}.
\item If $L=L_1L_2L_3L_4$ with $L_i$ of order $1$ for $i=1,2,3,4$ solve by variation of constant and return the solutions.
\item If there is a single factor $\tilde{L}$ of order $2$, then $L=L_1L_2L_3$ where $\tilde{L}=L_i$ for some $i\in\{1,2,3\}$ and the other two factors are of order $1$. Apply Kovacic algorithm to $\tilde{L}$ \cite{ULMER1996}.
\begin{enumerate}
\item If $\tilde{L}$ is solvable, then solve $L$ by variation of constants.
\item If $\tilde{L}$ is not solvable, compute hyperexponential solutions of $L$ \cite{PFLUGEL1997}.
\begin{enumerate}
\item If there is one, then $L=ML_0$ where $M$ is of order $3$ and $L_0$ is of order $1$. Compute hyperexponential solution of $M$ \cite{PFLUGEL1997}.
\begin{enumerate}
\item If there is one, then $L=NM_0L_0$ where $N$ is of order $2$ and $M_0$ is of order $1$. Solve $M_0L_0$ by variation of constants and return their solutions.
\item If there are no hyperexponential solutions to $M$, return a solution of $L_0$.
\end{enumerate}
\item If there are no hyperexponential solutions to $L$, return $[]$.
\end{enumerate}
\end{enumerate}
\item If there are two factors of order $2$, $L=L_1L_2$, apply Kovacic algorithm \cite{ULMER1996} to them
\begin{enumerate}
\item If $L_1$ is not solvable, return the Liouvillian solutions of $L_2$.
\item If $L_1$ and $L_2$ are solvable compute the solution through variation of constants
\item If $L_1$ is solvable but $L_2$ is not, compute an LCLM factorization of $L$ \cite{VANHOEIJ1997}. If it has two factors, solve them by Kovacic algorithm and return the Liouvillian solutions. Else return $[]$.
\end{enumerate}
\item Else $L$ is irreducible. Compute linearly independent projective Poisson structures of $L$
\begin{enumerate}
\item If there are less than $2$, return $[]$.
\item Else denote two projective Poisson structures by $W_1$,$W_2$, such that $W_1$ has rational coefficients and $W_2$ has coefficients in a qua\-dra\-tic extension. Let $w(z)\in\mathbb{K}(z)$ be such that the coefficients of $W_2$ are in $\mathbb{K}(z,\sqrt{w(z)})$.
\begin{enumerate}
\item Solve $\hbox{Pf}(W_1+\lambda W_2)=0$, and compute the two conjugate kernels $V_1,V_2$ with coefficients in $\mathbb{K}(z,\sqrt{w(z)})$ of $W_1+\lambda W_2$ for the two solutions $\lambda$.
\item Compute the companion differential system associated to $L$ restricted to $V_1$. Use a cyclic vector to compute an operator $\tilde{L}$ of order two with coefficients in $\mathbb{K}(z,\sqrt{w(z)})$.
\begin{enumerate}
\item If the symmetric square of $\tilde{L}$ has non-trivial solutions in $\mathbb{K}(z,\sqrt{w(z)})$, compute two linearly independent Liouvillian solutions to $\tilde{L}$ \cite{ULMER1996} of the form
$$e^{\int\sqrt{\alpha(z)+\sqrt{w(z)}\beta(z)}dz}.$$
\item Else, if for some $i\in\{6,8,12\}$ the $i$-th symmetric power of $\tilde{L}$ has non-trivial solutions in $\mathbb{K}(z,\sqrt{w(z)})$, use van Hoeij-Weil algorithm \cite{VANHOEIJ2005} to compute a basis of solutions to $\tilde{L}$ of the form
$$e^{\int \alpha(z)+\sqrt{w(z)}\beta(z) dz} F(p(z)+\sqrt{w(z)}r(z))$$
where $F$ is a solutions of a standard equation.
\item Else $\tilde{L}$ is not solvable.
\end{enumerate}
\end{enumerate}
Return the Liouvillian solutions to $\tilde{L}$ obtained in ii. together with their conjugate $\sqrt{w(z)} \rightarrow -\sqrt{w(z)}$.
\end{enumerate}
\end{enumerate}

\begin{thm}
The algorithm \underline{\sf SymplecticKovacic} returns the $\;$ vector space of Liouvillian solutions of $L$.
\end{thm}

\begin{proof}
According to Proposition \ref{propgroups} the possible factorizations of $L$ are for factors with orders:

i) $1,1,1,1$, ii) $1,1,2$, iii) $1,2,1$, iv) $2,1,1$, v) $2,2$, vi) $4$.

The first step tests whether we are in one of the first four cases.

The second step deals with case i), where we can obtain a basis of the space of Liouvillian solutions, which has dimension $4$, by using variation of constants.

The third step tests whether we are in cases ii), iii) or iv). If the order $2$ factor $\tilde{L}$ is solvable, we can obtain a basis of the space of Liouvillian solutions, which has dimension $4$, by using variation of constants after finding Liouvillian solutions to the factor of second order. If $\tilde{L}$ is not solvable, step 3(b) test whether we are in case iii), in which case the space of Liouvillian solutions has dimension two, and a basis can be obtain by variation of constants, or in case ii), in which case the space of Liouvillian solutions has dimension one, and a basis can is obtain by computing hyperexponential solutions to $L$. If we are not in either of these two cases, the only possible factorization of $L$ is of the form case ii) and there are no Liouvillian solutions, for the order two factor $\tilde{L}$ is not solvable.

The fourth step addresses  the case $v)$. If only the right factor is solvable, then the space of Liouvillian solution has dimension two and using Kovacic algorithm one can find a basis of it. If both factors are solvable, then the space of Liouvillian solution has dimension four and using variation of constants one can find a basis of it, after using Kovacic algorithm on both factors. If only the left factor is solvable, then step 4(c) test whether the Galois group has a representation $2\times 2$ block diagonal. In which case the space of Liouvillian solution has dimension two, as one of the two factors of the LCLM factorization is solvable. 

The fifth step deals with the irreducible case, which is case v). According to Corollary \ref{cor} ii), $L$ is not solvable if $L$ does not admit at least two linearly independent projective Poisson structures in a quadratic extension. Step 5(a) rules this case out. When there are two linearly independent projective Poisson structures in a quadratic extension, Corollary \ref{cor} iii), implies that a linear combination of them will produce two non-trivial and non-invertible Poisson structures. These Poisson structures are obtained in Step 5(b)i. Their kernels are two invariant vector spaces $V_1,V_2$ of solutions of $L$. None of these kernels $V_1$ and $V_2$ is rational, or else $L$ would possess an invariant vector space of dimension $2$, and thus would factorize. Therefore $V_1$ and $V_2$ are conjugate by an automorphism of a quadratic extension. In particular the solutions in $V_1$ are Liouvillian if and only if the solution in $V_2$ are Liouvillian. Step 5(b)ii computes the operator $L$ restricted to $V_1$, tests using Kovacic algorithm whether this restriction is solvable, and in such case computes its Liouvillian solutions. The solutions of $V_2$ when they are Liouvillian can be obtained by taking the conjugates of the Liouvillian solutions in $V_1$. 
\end{proof}

\section{Complexity and examples}

The complexity of computing rational solutions of differential operator does not depend only on the degree and order of the operator, but also on the local exponents at the singularities  \cite{BARKATOU1999}. The same applies to operator factorization \cite{VANHOEIJ1997}, and therefore also to our work. Because of this, we do not aim at obtaining low complexity in terms of the degree, order and coefficients height. The objective is more to obtain a workable algorithm on reasonable examples rather than on cases of worst complexity. The most expensive part is in the three last steps where finite groups are tested and so our algorithm was designed to optimize this part.
\begin{itemize}
\item The factorization in a quadratic extension allows to compute symmetric powers of an order $2$ operator instead of $4$
\item The finite Galois groups are tested using invariant instead of semi-invariants. Even if using semi invariant would allow lower order operators, the search of hyperexponential solutions in a quadratic extension is equivalent to searching hyperexponential solutions of an operator of order $14$, which is too expensive.
\item  The rational solutions in $\mathbb{K}(z,\sqrt{w(z)})$ are searched by first constructing a differential system, which is cheap, and secondly by constructing a universal denominator for each of the two unknowns by computing possible local exponents at the singularities. The rational solutions are then finally obtained using linear algebra.
\end{itemize} 

We present a table of examples obtained by considering the annihilator of $f(\sqrt{z})$ where $f$ is a solution of $f''+(z^{2n+1}+z+1)f=0$. This produces a symplectic operator of order $4$, which is bi-projectively symplectic as the Galois group will be a subgroup of $\mathbb{Z}_2 \ltimes SL_2(\mathbb{K})$. This also ensures that the algorithm do not use early termination paths to avoid the last steps. The computations were done on a Macbook pro 2013 2.8 Ghz.
\begin{center}
\begin{tabular}{|c|c|c|c|c|c|}\hline
$n$ & 0 &  1 & 2 & 3 & 4 \\\hline 
degree & 3 & 7 & 11 & 15 & 19\\\hline
time& 2.4 s &  5 s  & 9.2s  & 37.5 s & 5757s \\\hline
\end{tabular}
\end{center}

We present now the solutions of several symplectic operators with finite Galois groups. The pullbacks are obtained using the formulas in \cite{VANHOEIJ2005}, where the pullback functions are in a quadratic extension.

\textbf{A $D_8$ example}. We consider the LCLM of the following operator with its conjugate
$$Dz^2+\frac{3(20z+37\sqrt{z}+21}{256z^2(\sqrt{z}+1)^2}.$$
The solutions are found in $1.7s$:
$$\sqrt{z} (1\pm\sqrt{z})^{\frac{1}{4}}e^{\frac{1}{16}\int \frac{1}{z\sqrt{1\pm\sqrt{z}}} dz}, \sqrt{z} (1\pm\sqrt{z})^{\frac{1}{4}}e^{-\frac{1}{16}\int \frac{1}{z\sqrt{1\pm\sqrt{z}}} dz}$$
\indent \textbf{An $A_5$ example}. We consider the LCLM of the following operator with its conjugate
$$Dz^2+\frac{1}{2z}Dz+\frac{739z^{3/2}+864z^2+611\sqrt{z}-314z+800}{14400z^2(z-1)^2}.$$
The solutions are found in $12.4s$
{\scriptsize $$\sqrt[12]{\frac{z^2P(z)(\sqrt{z}-1)^2}{(5589\sqrt{z}-800)^3}}\mathcal{L}\left(-\frac{1}{6},5,\sqrt{99 \frac{(27945z-19967\sqrt{z}+1600)^2}{(5589\sqrt{z}-800)^3(1-\sqrt{z})}}\right)$$}
together with the conjugates $\sqrt{z} \mapsto -\sqrt{z}$, where $\mathcal{L}$ a basis a solutions of the Legendre differential equation (whose solutions can also be written in terms of the hypergeometric function) and
{\scriptsize $$P= 251894530944z^2-360031369239z^{3/2}+134021894211z-$$
$$17568425600\sqrt{z}+765440000 $$}
\indent \textbf{An $A_4$ example}. We consider the LCLM of the following operator with its conjugate
$$Dz^2+\frac{108z^2+648z^{3/2}+1505z+1498\sqrt{z}+560}{576(\sqrt{z}+1)^2z^2(2+\sqrt{z})^2}.$$
The solutions are found in $17s$
{\scriptsize $$\left(\frac{(189z^2+810z^{\frac{3}{2}}+1118z+526\sqrt{z}+20)^{24}z^{10}}{P(z)Q(z)^{14}(2+\sqrt{z})^6(\sqrt{z}+1)^6}\right)^{\frac{1}{24}}$$
$$\;_2F_1\left(\frac{13}{24},\frac{25}{24},\frac{5}{4},\frac{P(z)}{45(z+3\sqrt{z}+2)^2Q(z)^2} \right)$$}
with
{\scriptsize $$P=67191201z^6+863886870z^{11/2}+4900709061z^5+16136882532z^{9/2}+$$
$$34114858452z^4+48314544768z^{7/2}+46335734636z^3+29648385408z^{5/2}\! +$$
$$12093966336z^2+2856633184z^{3/2}+318081360z+10315200\sqrt{z}+104000$$
$$Q=945z^2+3240z^{3/2}+3354z+1052\sqrt{z}+20$$}
\indent \textbf{A $S_4$ example}. We consider the LCLM of the following operator with its conjugate
$$Dz^2+\frac{4z^{9/2}+4z^{3/2}+4z+3}{16z^2}$$
The solutions are found in $9.3s$
{\scriptsize $$((\sqrt{z}+1)^3z^{7/2}(252z+311\sqrt{z}+63)^{-3}(144027072z^3+534597840z^{5/2}+$$
$$774164272z^2\! +550356683z^{3/2}\! +198862573z+34349049\sqrt{z}+2250423))^{1/8}$$
$$\mathcal{L}\left(-\frac{1}{4},\frac{1}{3},\sqrt{-\frac{28(1512z^{3/2}+2815z+1496\sqrt{z}+189)^2}{5(252z+311\sqrt{z}+63)^3}} \right) $$}
\indent \textbf{An application in Hamiltonian systems}. We consider the potential
$$V=q_1q_2^2+q_1q_2q_3+q_1q_3^2+q_1^2+q_2^2+q_3^2.$$
This potential admits a particular solution along the line $q_2=q_3=0$. The variational equation is given by
$$\left(\begin{array}{c} \ddot{X}\\ \ddot{Y} \\ \ddot{Z} \end{array} \right)=  \left(\begin{array}{ccc} -2&0&0\\ 0&-2q_1(t)-2& -q_1(t) \\ 0&-q_1(t) & -2q_1(t)-2 \\\end{array} \right)
\left(\begin{array}{c} X\\ Y \\ Z \end{array} \right)$$
Now making a time change by denoting $q_1(t)=z$, we obtain in the lower right $2\times 2$ invariant block
$$(1-z^2)\left(\begin{array}{c} \ddot{Y} \\ \ddot{Z}\\ \end{array} \right)-z\left(\begin{array}{c} \dot{Y} \\ \dot{Z}\\ \end{array} \right)+\left(\begin{array}{cc}  2(z+1)& z \\ z & 2(z+1) \\\end{array} \right) \left(\begin{array}{c} Y \\ Z\\ \end{array} \right)=0$$
Applying cyclic vector on $Y(z)$ we obtain the following order $4$ operator $L=$
$$Dz^4+\frac{2(2z^2+1)}{z(z^2-1)}Dz^3-\frac{4z^5+2z^4-4z^3-3z^2-2}{z^2(z^2-1)^2}Dz^2-$$
$$\frac{4(z^2-z+1)}{(z^2-1)(z-1)z}Dz+\frac{3z^4+8z^3+2z^2+4}{z^2(z^2-1)^2}$$
This operator admits $2$ linearly independent projective symplectic structures of same type:
{\scriptsize $$ \frac{(z^2-1)^{3/2}}{z^2}\left( \begin {array}{cccc} 0&-4\,{\frac {z}{z-1}}&{\frac {2\,{z}^{2}
+1}{{z}^{2}-1}}&z\\ \noalign{\medskip}4\,{\frac {z}{z-1}}&0&-z&0
\\ \noalign{\medskip}-{\frac {2\,{z}^{2}+1}{{z}^{2}-1}}&z&0&0
\\ \noalign{\medskip}-z&0&0&0\end {array} \right),$$
$$\frac{(z^2-1)^{3/2}}{z^2}\left( \begin {array}{cccc} 0&-{\frac {3\,{z}^{2}-2}{{z}^{2}-1}}&-6\,
{\frac {z}{{z}^{2}-1}}&-2\\ \noalign{\medskip}{\frac {3\,{z}^{2}-2}{{z
}^{2}-1}}&0&{\frac {4\,{z}^{2}-1}{{z}^{2}-1}}&z\\ \noalign{\medskip}6
\,{\frac {z}{{z}^{2}-1}}&-{\frac {4\,{z}^{2}-1}{{z}^{2}-1}}&0&{z}^{2}-
1\\ \noalign{\medskip}2&-z&-{z}^{2}+1&0\end {array} \right)
$$}
It is not symplectic despite coming from a Hamiltonian system due to the algebraic change of variable. Indeed the base field we should consider here contains $\sqrt{1-z^2}$, but as it does not appear in the coefficients, we ``forget it'' in the computations. This adds a finite extension on the Galois group as the base field becomes smaller, but it has no consequences on the Liouvillian solvability of it. Actually algebraic changes of variable never pose a problem to our algorithm, as the resulting operator will always have a Galois group in $PSP_4(\mathbb{K})$. 

\begin{prop}\label{prop6}
The Lie groups of $GL_{4}(\mathbb{K})$ which are finite extensions of $SP_{4}(\mathbb{K})$ are of the form
$$\{M\in M_{2n}(\mathbb{K}), \; \exists \lambda\in\mathbb{U}_p, \; M^t JM=\lambda J\} \subset PSP_4(\mathbb{K})$$
where $\mathbb{U}_p$ is the set of $p$-roots of unity.
\end{prop}

\begin{proof}
Using the classification of \cite{HARTMANN2002}, we look for Lie group whose identity component are $SP_4(\mathbb{K})$. In $SL_4(\mathbb{K})$ these groups are diagonal extensions of $SP_4(\mathbb{K})$ with an additional generator of the form $\omega I_4$ with $\omega^4=1$. The group $GL_4(\mathbb{K})$ is itself a diagonal extension of $SL_4(\mathbb{K})$ and thus finite extensions of $SP_4(\mathbb{K})$ in $GL_4(\mathbb{K})$ are obtained by adding a generator of the form $\omega I_4$ where $\omega$ is a root of unity. 
\end{proof}

Applying the algorithm  \underline{\sf SymplecticKovacic} returns $[]$, meaning that there are no Liouvillian solutions, and thus that the potential is not integrable \cite{MORALES1999}. Still, the existence of two symplectic structures of the same exponential type allows us to build two linearly independent Poisson structures, and thus implies that the operator is the LCLM of two operators of order $2$:
$$Dz^2+\frac{zDz}{z^2-1}-\frac{3z+2}{z^2-1},\; Dz^2+\frac{zDz}{z^2-1}-\frac{z+2}{z^2-1}.$$

\textbf{Conclusion}. We produced an algorithm for computing Liouvillian solutions of symplectic operators of order $4$ whose execution time is manageable on reasonable examples. The key point is that the symplectic group admits much less primitive/imprimitive finite groups than $SL_4(\mathbb{K})$, allowing the i\-den\-ti\-fi\-ca\-tion of solvable Galois groups to be much faster. This allows in particular to study irreducible order $4$ operators which where before intractable. One of the application is the study of variational equations in Hamiltonians with $3$ degrees of freedom, which was impossible before as generally the variational equation has its Galois group in $PSP_4(\mathbb{K})$ and so it does not necessarily decouple into smaller order operators. Possible future enhancements include early termination by exponent analysis, optimization of the universal denominator, using the Poisson structures to speed up the factorization in reducible cases, optimization of the pullback expressions for the finite case representation, minimization of iterated integrals in the representation of the solutions, and a better handing of quadratic extensions. The code is available at \url{http://combot.perso.math.cnrs.fr/software.html}.


\bibliographystyle{plain}
\bibliography{SK}

\begin{thebibliography}{10}

\bibitem{APARICIO2013}
Ainhoa Aparicio-Monforte, Elie Compoint, and Jacques-Arthur Weil.
\newblock A characterization of reduced forms of linear differential systems.
\newblock {\em J. Pure Appl. Algebra}, 217(8):1504--1516, 2013.

\bibitem{BARKATOU1999}
Moulay~A. Barkatou.
\newblock On rational solutions of systems of linear differential equations.
\newblock {\em J. Symbolic Comput.}, 28(4-5):547--567, 1999.
\newblock Differential algebra and differential equations.

\bibitem{BURGER2004}
Reinhold Burger, George Labahn, and Mark van Hoeij.
\newblock Closed form solutions of linear {ODE}s having elliptic function
  coefficients.
\newblock In {\em I{SSAC} 2004}, pages 58--64. ACM, New York, 2004.

\bibitem{HANANY2001}
Amihay Hanany and Yang-Hui He.
\newblock A monograph on the classification of the discrete subgroups of {$\rm
  SU(4)$}.
\newblock {\em J. High Energy Phys.}, (2):Paper 27, 12, 2001.

\bibitem{HARTMANN2002}
Julia Hartmann.
\newblock Invariants and differential {G}alois groups in degree four.
\newblock In {\em Differential {G}alois theory ({B}polhkedlewo, 2001)},
  volume~58 of {\em Banach Center Publ.}, pages 79--87. Polish Acad. Sci. Inst.
  Math., Warsaw, 2002.

\bibitem{VANHOEIJ2005}
Mark~van Hoeij and Jacques-Arthur Weil.
\newblock Solving second order linear differential equations with {K}lein's
  theorem.
\newblock In {\em I{SSAC}'05}, pages 340--347. ACM, New York, 2005.

\bibitem{KOVACIC1986}
Jerald~J. Kovacic.
\newblock An algorithm for solving second order linear homogeneous differential
  equations.
\newblock {\em J. Symbolic Comput.}, 2(1):3--43, 1986.

\bibitem{MORALES1999}
Juan~J. Morales~Ruiz.
\newblock {\em Differential {G}alois theory and non-integrability of
  {H}amiltonian systems}.
\newblock Modern Birkh\"auser Classics. Birkh\"auser/Springer, Basel, 1999.
\newblock [2013] reprint of the 1999 edition [MR1713573].

\bibitem{PFLUGEL1997}
E.~Pfl\"ugel.
\newblock An algorithm for computing exponential solutions of first order
  linear differential systems.
\newblock In {\em Proceedings of the 1997 {I}nternational {S}ymposium on
  {S}ymbolic and {A}lgebraic {C}omputation ({K}ihei, {HI})}, pages 164--171.
  ACM, New York, 1997.

\bibitem{VANDERPUT2003}
Marius van~der Put and Michael~F. Singer.
\newblock {\em Galois theory of linear differential equations.}
\newblock Grundlehren der mathematischen Wissenschaften: 328. Berlin ; New York
  : Springer, c2003., 2003.

\bibitem{ULMER1996}
Felix Ulmer and Jacques-Arthur Weil.
\newblock Note on {K}ovacic's algorithm.
\newblock {\em J. Symbolic Comput.}, 22(2):179--200, 1996.

\bibitem{VANHOEIJ1997}
Mark van Hoeij.
\newblock Factorization of differential operators with rational functions
  coefficients.
\newblock {\em J. Symbolic Comput.}, 24(5):537--561, 1997.

\end{thebibliography}

\end{document}